\documentclass[12pt,reqno,a4paper]{amsart}
\usepackage{amsmath,amssymb,amsthm,amsfonts}
\usepackage{tikz}
\usepackage{caption}
\usepackage{subcaption}

\newtheorem{theorem}{Theorem}
\newtheorem{lemma}[theorem]{Lemma}
\newtheorem{proposition}[theorem]{Proposition}
\newtheorem{corollary}[theorem]{Corollary}
\newtheorem{remark}[theorem]{Remark}
\newtheorem{defn}[theorem]{Definition}

\newcommand{\Rr}{\mathbb{R}}

\newcommand\minus{%
  \setbox0=\hbox{-}%
  \vcenter{%
    \hrule width\wd0 height \the\fontdimen8\textfont3%
  }%
}

\begin{document}

\title[$\lambda$-stability]
{$\lambda$-stability of periodic billiard orbits}

\author[Gaiv\~ao]{Jos\'e Pedro Gaiv\~ao}
\address{Departamento de Matem\'atica and CEMAPRE, ISEG\\
Universidade de Lisboa\\
Rua do Quelhas 6, 1200-781 Lisboa, Portugal}
\email{jpgaivao@iseg.ulisboa.pt}

\author[Troubetzkoy]{Serge Troubetzkoy}
\address{Aix Marseille Universit\'e, CNRS\\ Centrale Marseille, I2M, UMR 7373}
\email{serge.troubetzkoy@univ-amu.fr}

\date{\today}  

\begin{abstract}
We introduce a new notion of stability for periodic orbits in polygonal billiards. We say that a periodic orbit of a polygonal billiard is $\lambda$-stable if there is a periodic orbit for the corresponding pinball billiard which converges to it as $\lambda\to1$. This notion of stability is unrelated to the notion introduced by Galperin, Stepin and Vorobets. We give sufficient and necessary conditions for a periodic orbit to be $\lambda$-stable and prove that the set of $d$-gons having at most finite number of $\lambda$-stable periodic orbits is dense is the space of $d$-gons. Moreover, we also determine completely the $\lambda$-stable periodic orbits in integrable polygons. 
\end{abstract}
 
\maketitle
\tableofcontents

\section{Introduction}
The billiard in a polygon $P$ is the flow on the unit tangent bundle of $P$ corresponding to the free motion of a point-particle with specular reflection at $\partial P$, i.e., the angle of reflection equals the angle of incidence.  The first return map of the billiard flow to the boundary is called the billiard map. Polygonal billiards are known to have zero topological entropy, thus zero Lyapunov exponents \cite{GKT,G-H,Katok}. 

Recently, a new class of billiards with non-standard reflection laws, known as pinball billiards, has been introduced in which the reflection law contracts the angle of reflection by a factor $\lambda\in(0,1)$ towards the normal of the table \cite{M-P-S}. In this case, the billiard map no longer preserves the Liouville measure and the system may have hyperbolic attractors \cite{DelMagno-et-al}. The inverse of the pinball billiard map with contraction $\lambda$ is conjugated to a billiard map which expands the angle of reflection away from the normal of table by a factor $\lambda^{-1}$. This means that the hyperbolic attractors for the billiard which contracts angles by $\lambda$ become hyperbolic repellers for the billiard which expands angles by $\lambda^{-1}$. For this reason we also call the billiard expanding angles by $\lambda^{-1}$ a pinball billiard.

Perpendicular orbits bouncing between parallel sides are parabolic and are exactly the periodic orbits of period two for the billiard map. If the billiard in $P$ has no such orbits, then for any $\lambda\in(0,1)$, the pinball billiard in $P$ has a finite number of ergodic Sinai-Ruelle-Bowen (SRB) measures, each supporting a uniformly hyperbolic attractor. The union of the basins of attraction of the SRB measures has full Lebesgue measure \cite{DelMagno-et-al-B}, and by Pesin theory, the periodic orbits are dense in the support of the SRB measures. 

Therefore, the billiard of any polygon can be embedded in a one-parameter family, parametrized by $\lambda\in(0,+\infty)$, of billiards having periodic orbits whenever $\lambda\neq1$. In the conservative limit $\lambda\to1$ we obtain the billiard with the standard reflection law for which the existence of periodic orbits is an open problem. Hence, it would be interesting to characterize the periodic orbits of the billiard in $P$ that can be approximated by periodic orbits of the corresponding pinball billiard.

In this paper we introduce a new notion of stability for periodic orbits. We say that a periodic orbit $q$ of the billiard in $P$ is \textit{$\lambda$-stable} if there exist a sequence $\lambda_n\to 1$ of positive real numbers different from 1 and periodic orbits $q_{\lambda_n}$ for the pinball billiard in $P$ hitting the same sides of $P$ such that $q_{\lambda_n}$ converges to $q$ as $n\to \infty$. In fact, we define two notions of stability, which we call $\lambda^{\pm}$-stable, depending if the sequence $\lambda_n$ converges to $1$ from above or below. We study these notions of stability and show that they capture the same periodic orbits, i.e., $\lambda^-$-stability and $\lambda^+$-stability coincide. We also show that $\lambda$-stability is unrelated to the notion of stability for periodic orbits introduced by Galperin, Stepin and Vorobets \cite{GSV}. Moreover, we derive sufficient and necessary conditions for a periodic orbit to be $\lambda$-stable, and using these conditions 
we prove that the set of $d$-gons having at most finite number of $\lambda$-stable periodic orbits is dense is the space of $d$-gons. Finally, 
we determine the $\lambda$-stable periodic orbits in integrable polygons, i.e., the rectangles, equilateral triangle and right triangles 45-45-90 and 30-60-90 degrees.

The rest of the paper is structured as follows. In section~\ref{section preliminaries} we introduce the pinball billiard and prove some preliminary results
on $\lambda$-stability. In section~\ref{nec and suf section} we derive sufficient and necessary conditions for a periodic orbit to be $\lambda$-stable and in section~\ref{sec:results} we collect some applications to rational and integrable polygons. Finally, section~\ref{sec:proofs} contains the proofs of the results stated in section~\ref{sec:results}. 

\section{Preliminaries}\label{section preliminaries}

\subsection{Polygonal billiards}
For simplicity we assume that $P$ is a simply connected $d$-gon, but our results hold for finitely connected polygons as well. 
 Give $P$ the anti-clockwise orientation. 
Let $M$ be the set of unit tangent vectors $(x,v)$ of $P$ such that $x\in\partial P$ and $v$ is pointing inside $P$. A vector $(x,v)$ defines an oriented ray and we denote by $(\bar{x},\bar{v})$ the unit tangent vector where $\bar{x}$ is the first intersection of the ray with $\partial P$ and $\bar{v}$ is the vector pointing inside $P$ obtained by reflecting $v$ along the boundary of $P$. 
The part of the ray connecting $x$ to $\bar{x}$ is called a \emph{link} of the trajectory and will be noted
$((x,v),\bar{x})$.
The map $\Phi:M\to M$ just defined is called the \textit{billiard map}. Notice that the billiard map is not continuous for the vectors $(x,v)$ whose foot point image $\bar{x}$ is a vertex of $P$. So we exclude those vectors from $M$ as well as the vectors of $M$ sitting at vertices of $P$. We denote this subset of $M$ by $M'$. Labeling the sides of $P$ from $1$ to $d$, we see that the domain $M'$ of the billiard map is a finite union of connected components $\Sigma_{i,j}$ which correspond to billiard segments that connect side $i$ to side $j$. 

Consider a billiard trajectory starting at $x$ in the direction $v$ and its first link $((x,v),\bar{x})$.
To obtain the direction of the second link we reflect this link in the boundary of $P$. Instead 
we reflect $P$ about the side of $P$ containing $\bar{x}$, the reflection of the second link 
becomes the straightened continuation of the first link.  The iteration of this procedure is called
\emph{unfolding} or \emph{straightening} of the billiard trajectory.

 A \emph{generalized diagonal} is a billiard trajectory that connects two corners of $P$, more formally
 it is a finite orbit $(\Phi^i(x,v): i = 0, \dots, n)$ such that $x$ is a corner of $P$ and the image of $\Phi^n(x,v)$
is not defined since the  first intersection of the oriented ray in this direction with $\partial P$ is a corner
of $P$ (possibly the same corner).

If we parametrize $\partial P$ by the arc-length parameter $s\in S^1$ and denoting by $x(s)\in\partial P$ the corresponding point in $\partial P$, then we can give coordinates $(s,\theta)$ to a point in $(x(s),v)\in M'$ where $\theta$ is the oriented angle between $v$ and the inward normal at $x(s)$. The angle $\theta$ belongs to the interval $\left(-\pi/2,\pi/2\right)$. Denote by $(\bar{s},\bar{\theta})$ the image of $(s,\theta)$ under the billiard map. Let $\beta_{i,j}$ be equal to $\pi$ minus the angle formed by the oriented sides $i$ and $j$ of $P$. Then the angle $\bar{\theta}$ is given by,
\begin{equation}\label{angle map}
\bar{\theta}=\beta_{i,j}-\theta,
\end{equation}
and the derivative of the billiard map is, 
\begin{equation}\label{derivative}
	\begin{split}
		 D \Phi(s,\theta)
		 = - \begin{pmatrix}
				\dfrac{\cos \theta}{\cos \bar{\theta}} & \dfrac{t(s,\theta)}{\cos \bar{\theta}} \\[1em]
				0 & 1 \\
			\end{pmatrix},
	\end{split}
\end{equation}
where $t(s,\theta)$ denotes the length of the segment joining $x(s)$ and $x(\bar{s})$. For the computation of the derivative, see e.g. \cite[Formula~(2.26)]{Chernov-Markarian}.
%
%

 Periodic orbits for the billiard map $\Phi$ come in one-parametric families. Indeed, any periodic trajectory of period $p$ can be included in a \emph{cylinder} of periodic trajectories having the same period when $p$ is even or twice the period when $p$ is odd. Moreover, every periodic cylinder contains a generalized diagonal on its boundary.
We will see below that for pinball billiards the situation is completely different. 

\subsection{Pinball billiard map}\label{pinball section}

Given $\lambda>0$, let $R_\lambda:M\to M$ be the map $(s,\theta)\mapsto (s,\lambda\theta)$. Define
$$
\Phi_{\lambda}:=R_\lambda\circ\Phi.
$$ 
If $\lambda\neq 1$, then we call $\Phi_{\lambda}$ the \textit{pinball billiard map}. When $\lambda=1$, we recover the billiard map with the standard (elastic) reflection law. Let $S$ denote the involution map $(s,\theta)\mapsto(s,-\theta)$. Using $\Phi^{-1}\circ S=S\circ\Phi$ we conclude that
\begin{equation}\label{conjugacy}
\Phi^{-1}_{\lambda}\circ S=S\circ \Phi_{\lambda^{-1}},
\end{equation}
i.e., the billiard map $\Phi_\lambda$ is conjugated to $\Phi_{\lambda^{-1}}^{-1}$ by the involution $S$.

If $(s,\theta)\in\Sigma_{i,j}$ and $(s',\theta')=\Phi_\lambda(s,\theta)$, then we have by \eqref{angle map} and \eqref{derivative} that
\begin{equation}\label{pinball map}
s'=-A_{i,j}(\theta)s+B_{i,j}(\theta)\quad\text{and}\quad\theta'=\lambda(\beta_{i,j}-\theta),
\end{equation}
where 
$$
A_{i,j}(\theta):=\frac{\cos\theta}{\cos(\beta_{i,j}-\theta)},
$$
and $B_{i,j}$ is an analytic function, in fact rational trigonometric, depending only on the sides $i$ and $j$ of $P$. It is not difficult to derive an explicit expression for $B_{i,j}$, but in the following we shall not use it.

An orbit of the pinball billiard map $\Phi_\lambda$ is a sequence $\{(s_n,\theta_n)\}_{n\geq0}$ in $M'$ such that $(s_{n+1},\theta_{n+1})=\Phi_\lambda(s_n,\theta_n)$. Given an orbit $\{(s_n,\theta_n)\in M'\colon n\geq0\}$, we denote by $i_n\in\{1,\ldots,d\}$ the corresponding \textit{itinerary}, i.e., $(s_n,\theta_n)\in\Sigma_{i_{n},i_{n+1}}$ for every $n\geq 1$.

\subsection{Periodic orbits and $\lambda$-stability}

\begin{proposition}\label{prop-isolated}
When $\lambda\neq 1$, all periodic orbits for $\Phi_{\lambda}$ having period $p>2$ are hyperbolic, isolated and their total number does not exceed $d^p$.
\end{proposition}

\begin{proof}
By \eqref{conjugacy} it is sufficient to prove the statement assuming $\lambda<1$. Being hyperbolic follows from \cite{DelMagno-et-al}, which implies that are isolated. The number of periodic orbits follows from the fact that for any $p>2$ there are at most $d^p$ periodic itineraries of length $p$ (see the proof of Proposition~\ref{prop-unique}).
\end{proof}

%
%
%

\begin{defn}

We say that a periodic orbit $q$ of $\Phi$ is \textit{$\lambda^{\pm}$-stable} if there exist a strictly decreasing (increasing)  sequence $\{\lambda_n\}_{n\geq1}$ converging to 1 and for each $n\in\mathbb{N}$ a periodic orbit $q_n$ of $\Phi_{\lambda_n}$ such that $q$ and $q_n$ have the same itinerary and $q_n\to q$ as $n\to\infty$. We say that $q$ is \textit{$\lambda$-stable} if it is both $\lambda^+$-stable and $\lambda^-$-stable.

We say that a periodic cylinder of $\Phi$ is  \textit{$\lambda^{\pm}$-stable} if there exists a $\lambda^{\pm}$-stable periodic orbit in the cylinder, and \textit{$\lambda$-stable} if it is both $\lambda^+$-stable and $\lambda^-$-stable.

\end{defn}

We call a periodic orbit a \textit{ping-pong} orbit if it bounces between two parallel sides.
Clearly each ping-pong orbit is $\lambda$-stable, and thus the whole associated cylinder is
$\lambda$-stable. For all other periodic cylinders we have a different behavior. 

\begin{proposition}\label{prop-unique}
If a periodic cylinder is $\lambda^{\pm}$-stable and not a ping-pong cylinder
 then the associated $\lambda^\pm$-stable periodic orbit is unique.
\end{proposition}

\begin{proof}
This follows from the fact that for any given periodic itinerary there is at most one periodic orbit realizing that itinerary. Indeed, let $n>2$ denote the period of the cylinder. By \eqref{pinball map}, the second component of $\Phi^n_\lambda$ is a composition of affine maps of the form $\theta\mapsto \lambda(\beta_{i,j}-\theta)$. Thus, the angle $\theta_\lambda$ of the associated $\lambda^\pm$-stable periodic orbit is uniquely determined by its itinerary. Similarly, the first component of $\Phi^n_\lambda$ is a composition of affine maps of the form $s\mapsto -A_{i,j}(\theta)s+B_{i,j}(\theta)$. Therefore, it is also affine in $s$ with coefficients depending only on the itinerary of the cylinder and on $\theta_\lambda$. Hence, the position $s_\lambda$ of the associated $\lambda^\pm$-stable periodic orbit is also uniquely determined by the itinerary. 
\end{proof}

\begin{proposition}\label{prop-odd}
Every periodic orbit with odd period is $\lambda$-stable.
\end{proposition}

\begin{proof}
Denote by $\{i_0,i_1,\ldots,i_{2n}\}$ the itinerary of the periodic orbit. For $(s_0,\theta_0)\in \Sigma_{i_0,i_1}\cap\Phi_\lambda^{-1}(\Sigma_{i_1,i_2})\cap\cdots\cap\Phi_\lambda^{-2n}(\Sigma_{i_{2n},i_0})$ we have, by \eqref{derivative}, that
$$
D\Phi_\lambda^{2n+1}(s_0,\theta_0)=-\begin{pmatrix}A_{i_0,\ldots,i_{2n+1}}(\theta_0)&*\\0&\lambda^{2n+1}\end{pmatrix},
$$
where,
$$
A_{i_0,\ldots,i_{2n+1}}(\theta_0)=\frac{\cos\theta_0}{\cos\bar{\theta}_{2n}}\rho_\lambda(\bar{\theta}_{0})\cdots\rho_\lambda(\bar{\theta}_{2n-1}),
$$
$\bar{\theta}_j=\beta_{i_{j},i_{j+1}}-\theta_{j}$ and,
$$
\rho_\lambda(\theta):=\frac{\cos(\lambda\theta)}{\cos(\theta)}.
$$
Notice that $\rho_\lambda(\theta)\neq 1$ for every $\theta\in\left(-\frac\pi2,\frac\pi2\right)\setminus\{0\}$ and $\rho_\lambda(\theta)=1$ if and only if $\theta=0$ or $\lambda=1$. Thus, if $(\hat{s}_0,\hat{\theta}_0)$ is $2n+1$-periodic for $\Phi$, then $A_{i_0,\ldots,i_{2n+1}}(\hat{\theta}_0)=1$ when $\lambda=1$. This implies that 
$$
\left.D\Phi^{2n+1}_\lambda(\hat{s}_0,\hat{\theta}_0)\right|_{\lambda=1}=\begin{pmatrix}-1&*\\0&-1\end{pmatrix}.
$$ 
By the implicit function theorem, there exists $\delta>0$ and analytic functions $s_0:(1-\delta,1+\delta)\to\Rr$ and $\theta_0:(1-\delta,1+\delta)\to(-\pi/2,\pi/2)$ such that $(s_0(\lambda),\theta_0(\lambda))$ is a $2n+1$-periodic point for $\Phi_\lambda$ and $(s_0(1),\theta_0(1))=(\hat{s}_0,\hat{\theta}_0)$.
\end{proof}

Galperin, Stepin and Vorobets have introduced and studied another notion of stability; a periodic orbit is called \emph{stable} if it
is not destroyed by any small deformation of the polygon (\cite{GSV} p.18). They have shown that a sufficient and necessary condition for a periodic orbit with period $n>1$ to be stable is that the alternating sum of the letters of the itinerary word $W=i_0\ldots i_{n-1}$ is zero, i.e., 
$$
i_0-i_1+i_2-i_3+\cdots \equiv 0.
$$
Using this criterion they proved that odd length periodic orbits are stable (\cite{GSV} p.28). Thus,

\begin{corollary}
Odd length periodic orbits are both stable and $\lambda$-stable.
\end{corollary}

Even length periodic orbits, other than ping-pong orbits, behave quite differently.  We will discuss such orbits in Section \ref{sec:results} and will see that these notions of stability are unrelated. Before closing this section we see that the notions of $\lambda^+$ and $\lambda^-$ stability are related as follows. Given a periodic cylinder $\mathcal{C}$ of $\Phi$ we denote by $\mathcal{C}^{-1}$ the periodic cylinder of $\Phi$ obtained by reversing $\mathcal{C}$, i.e., by applying the involution $S$ defined in Subsection~\ref{pinball section}.  

\begin{lemma}\label{Cminus}
The periodic cylinder $\mathcal{C}$ is $\lambda^{-}$-stable if and only if the periodic cylinder $\mathcal{C}^{-1}$ is $\lambda^{+}$-stable.
\end{lemma}

\begin{proof}
Suppose that $\mathcal{C}$ is $\lambda^-$-stable. By definition, there exists a periodic orbit $q_n$ and a strictly increasing sequence $\lambda_n\nearrow1$ such that $q_n$ converges to a periodic orbit in $\mathcal{C}$ and it is a periodic orbit of $\Phi_{\lambda_n}$. By \eqref{conjugacy}, $\Phi_{\lambda_n^{-1}}$ is conjugated to $\Phi_{\lambda_n}^{-1}$ by the involution map $S$. Thus, $S(q_n)$ is a periodic orbit for $\Phi_{\lambda_n^{-1}}$ and converges as $n\to\infty$ to a periodic orbit in $\mathcal{C}^{-1}$. So, $\mathcal{C}^{-1}$ is $\lambda^+$-stable. The other direction of the proof is analogous. 
\end{proof}


\section{Characterization of $\lambda$-stability}\label{nec and suf section}

Let $\mathcal{C}$ be a periodic cylinder of $\Phi$ having even period. Denote by $2n$ the period of the cylinder, $\{i_0,\ldots,i_{2n-1}\}$ the corresponding itinerary and $\hat{\theta}_0,\ldots,\hat{\theta}_{2n-1}$ the angles of the cylinder along the itinerary. We call $\hat{\theta}_0$ the \textit{angle of departure} of $\mathcal{C}$. 

\subsection{Necessary condition}

The following identity is well known as it is a simple consequence of \eqref{angle map}.

\begin{lemma}\label{prop-angles}
A necessary condition for the existence of an even length periodic orbit with periodic $2n$ for $\Phi$ with itinerary $\{i_0,\ldots,i_{2n-1}\}$ is that,
\begin{equation}\label{e:alternating}
\sum_{k=0}^{2n-1} (-1)^k \beta_{i_{k},i_{k+1}}   = 0.
\end{equation}
\end{lemma}

Using this lemma we obtain the following necessary condition for $\lambda$-stability.

\begin{proposition}\label{necessary condition}
If the cylinder $\mathcal{C}$ is $\lambda^{-}$-stable or $\lambda^{+}$-stable, then its angle of departure equals
$$
\frac{1}{2n}\sum_{k=0}^{2n-1}(-1)^{k+1}k\beta_{i_k,i_{k+1}}.
$$
\end{proposition}

\begin{proof}
Any periodic point for $\Phi_\lambda$ having the itinerary $\{i_0,\ldots, i_{2n-1}\}$ has an angle at $i_0$ equal to,
$$
\frac{1}{\lambda^{2n}-1}\sum_{k=0}^{{2n}-1}(-\lambda)^{{2n}-k}\beta_{i_k,i_{k+1}}.
$$
When $\lambda\to 1$, the numerator in the previous fraction is
$$
\sum_{k=0}^{{2n}-1}(-1)^{{2n}-k}\beta_{i_k,i_{k+1}},
$$
which is zero by Lemma~\ref{prop-angles}. Thus, by L'H\^otipal's rule,
\begin{align*}
\lim_{\lambda\to1}\frac{\sum_{k=0}^{{2n}-1}(-\lambda)^{{2n}-k}\beta_{i_k,i_{k+1}}}{\lambda^{2n}-1}&= \lim_{\lambda\to1}\frac{\sum_{k=0}^{{2n}-1}(k-2n)(-\lambda)^{{2n}-k-1}\beta_{i_k,i_{k+1}}}{2n\lambda^{2n-1}}
\\
&= \frac{1}{{2n}}\sum_{k=0}^{{2n}-1}(-1)^{k+1}k\beta_{i_k,i_{k+1}}.
\end{align*}
\end{proof}



%

\subsection{Sufficient conditions}

For each pair of distinct sides $i$ and $j$ of $P$ let $\Phi_{i,j}$ be the restriction map of $\Phi$ to the domain $\Sigma_{i,j}$. Since $\Phi_{i,j}$ is affine in the first component (see \eqref{pinball map}), it can be extended to all $(s,\theta)\in\Rr\times(-\pi/2,\pi/2)$ provided $|\beta_{i,j}-\theta|<\pi/2$. Denoting by $\ell_i$ the supporting line of the $i$-th side $\sigma_i$ of $P$ and by $x_i$ the arc-length parametrization of $\ell_i$ obtained by extending the parametrization of $\sigma_i$, then the first component of $\Phi_{i,j}$ can be seen as the projection of the point $x_i(s)\in\ell_i$ onto the line $\ell_j$ along the angle $\theta$. Let $(s',\theta')=\Phi_{i,j}(s,\theta)$. Then by \eqref{derivative},
\begin{equation}\label{branch derivative}
\frac{\partial s'}{\partial s}=-\frac{\cos\theta}{\cos(\beta_{i,j}-\theta)}\quad\text{and}\quad\frac{\partial s'}{\partial\theta}=-\frac{t(s,\theta)}{\cos(\beta_{i,j}-\theta)},
\end{equation}
where $t(s,\theta)$ now represents the oriented length of segment joining the points $x_i(s)$ and $x_j(s')$ that lie on the oriented line defined by $(s,\theta)$. 

\bigskip 

Throughout the rest of this section we suppose that the period of the periodic cylinder $\mathcal{C}$ is greater than two and that its angle of departure $\hat{\theta}_0$ is given by the expression in Proposition~\ref{necessary condition}. 

Define the functions,
\begin{equation}\label{theta0}\begin{split}
\theta_0(\lambda)&:=\frac{1}{\lambda^{2n}-1}\sum_{k=0}^{{2n}-1}(-\lambda)^{{2n}-k}\beta_{i_k,i_{k+1}},\\
\theta_{k}(\lambda)&:=\lambda(\beta_{i_{k-1},i_{k}}-\theta_{k-1}(\lambda)),\quad k=1,\ldots,2n.
\end{split}
\end{equation}
Clearly, 
$$
\lim_{\lambda\to1}\theta_k(\lambda)=\hat{\theta}_k\quad\text{and}\quad\theta_{2n}(\lambda)=\theta_0(\lambda).
$$
Hence, we can find $\delta>0$ such that $\theta_k(\lambda)\in(-\pi/2,\pi/2)$ for every $\lambda\in(1-\delta,1+\delta)$ and every $k=0,\ldots,2n-1$. 
So, for each $k=1,\ldots,2n$, the function $F_{\mathcal{C},k}:\Rr\times(1-\delta,1+\delta)\to\Rr$ is well defined by the formula,
$$
F_{\mathcal{C},k}(s,\lambda):=\pi_1\circ R_\lambda\circ\Phi_{i_{k-1},i_{k}}\circ\cdots\circ R_\lambda\circ\Phi_{i_0,i_1} (s,\theta_0(\lambda)),
$$
where $\pi_1$ denotes the projection $(s,\theta)\mapsto s$ and $R_\lambda$ the map defined in subsection~\ref{pinball section}. Note that $F_{\mathcal{C},2n}(s,1)=s$ for every $s\in\Rr$. In order to simplify the notation we simply write $F_k$ in place of $F_{\mathcal{C},k}$. 


\begin{lemma}\label{F}
$F_{2n}(\cdot,\lambda)$ is an affine map, orientation-preserving and depends analytically on $\lambda$. Moreover, it is expanding (contracting) if and only if $\lambda<1$ ($\lambda>1$).
\end{lemma}

\begin{proof}
The map is affine because it is a composition of maps $R_\lambda\circ \Phi_{i,j}$ whose first component is affine in the variable $s$ (see \eqref{pinball map}). Since $\theta_0(\lambda)$ is analytic for $\lambda>0$ (has a removable singularity at $\lambda=1$), $F_{2n}(\cdot,\lambda)$ is also analytic in $\lambda$. Moreover, by \eqref{branch derivative} we have
\begin{align*}
\frac{\partial F_{2n}}{\partial s}(s,\lambda)&=A_{i_0,i_1}(\theta_0(\lambda))\cdots A_{i_{2n-1},i_{2n}}(\theta_{2n-1}(\lambda))\\
&=\rho_\lambda(\bar{\theta}_{0}(\lambda))\cdots\rho_\lambda(\bar{\theta}_{2n-1}(\lambda))
\end{align*}
where
$\bar{\theta}_j(\lambda)=\beta_{i_{j},i_{j+1}}-\theta_{j}(\lambda)$ and $\rho_\lambda$ is the function defined in the proof of Proposition~\ref{prop-odd}. Because, all $\theta_k(\lambda)\in(-\pi/2,\pi/2)$, we have $\rho_\lambda(\bar{\theta}_k(\lambda))>0$ for every $k=0,\ldots,2n-1$. Thus, $F_{2n}(\cdot,\lambda)$ is orientation-preserving. Clearly,  $\lambda<1$ if and only if $\rho_\lambda(\theta)> 1$ for every $\theta\in\left(-\frac\pi2,\frac\pi2\right)\setminus\{0\}$. Since the period of the cylinder is greater than two, $\rho_\lambda(\bar{\theta}_k(\lambda))>1$ for some $\bar{\theta}_k$.  Thus, $F_{2n}(\cdot,\lambda)$ is expanding whenever $\lambda<1$. A similar argument shows that $F_{2n}(\cdot,\lambda)$ is contracting whenever $\lambda>1$.
\end{proof}

Let $I_\mathcal{C}$ denote the base of the cylinder $\mathcal{C}$, i.e.,the set of parameters $s\in \Rr$ such that $(s,\hat{\theta}_0)$ is a periodic point of $\Phi$ having the same itinerary of the cylinder. The elements of $I_\mathcal{C}$ are called \textit{base foot points}. Note that $I_\mathcal{C}$ is an open interval. 

Using the previous lemma we see that $\lambda^-$ and $\lambda^+$ stability notions coincide. 

\begin{proposition}\label{lambda plus minus}
If $\mathcal{C}$ is $\lambda^{-}$-stable or $\lambda^{+}$-stable, then $\mathcal{C}$ is $\lambda$-stable.
\end{proposition}

\begin{proof}
Suppose that $\mathcal{C}$ is $\lambda^-$-stable. The other case is analogous.
Let $s_0(\lambda_n)\in I_\mathcal{C}$ denote the base foot point corresponding to a periodic orbit $q(\lambda_n)$ of $\Phi_{\lambda_n}$ which converges to a periodic orbit in $\mathcal{C}$ as $n\to\infty$. By Proposition~\ref{prop-unique},  the periodic orbit $q(\lambda_n)$ is unique. In fact, $s_0(\lambda_n)=F_{2n}(s_0(\lambda_n),\lambda_n)$ for every $n\geq1$. Because $F_{2n}(\cdot,\lambda)$ is affine and analytic in $\lambda$, we can analytically extend $s_0$ to every
 $\lambda\in(1-\varepsilon,1+\varepsilon)$ for some $\varepsilon>0$. 
This extension gives a periodic orbit $q(\lambda)$ of $\Phi_{\lambda}$ for every $\lambda\in(1-\varepsilon,1+\varepsilon)$. Thus $\mathcal{C}$ is $\lambda$-stable.
\end{proof}

\begin{remark}
In fact, in the proof of Proposition~\ref{lambda plus minus}, we show that a periodic cylinder $\mathcal{C}$ is $\lambda^{\pm}$-stable if and only if there exists $\varepsilon>0$ such that for every $\lambda\in(1-\varepsilon,1+\varepsilon)$ the map $\Phi_\lambda$ has a periodic orbit $q_\lambda$ having the same itinerary of $\mathcal{C}$ such that $\lim_{\lambda\to1}q_\lambda$ is contained in the cylinder $\mathcal{C}$.
\end{remark}

As an immediate consequence of Lemma~\ref{Cminus} and Proposition~\ref{lambda plus minus} we obtain the following result.

\begin{corollary}\label{C and C minus}
$\mathcal{C}$ is $\lambda$-stable if and only if $\mathcal{C}^{-1}$ is $\lambda$-stable.
\end{corollary}

The following result gives a sufficient and necessary condition for $\lambda$-stability. 


\begin{proposition}\label{lem: F}
The periodic cylinder $\mathcal{C}$ is $\lambda$-stable if and only if there are $a,b\in I_\mathcal{C}$ with $a<b$ and $0<\lambda_0<1$ such that
$$
F_{2n}(a,\lambda)< a\quad\text{and}\quad F_{2n}(b,\lambda)> b,\quad\forall\,\lambda\in(\lambda_0,1).
$$
Moreover, if the above condition holds, then the $\lambda$-stable periodic orbit contained in $\mathcal{C}$ has base foot point $s\in(a,b)$.
\end{proposition}

\begin{proof}
By Lemma~\ref{F}, $F_{2n}(\cdot,\lambda)$ is affine, orientation-preserving and expanding for $\lambda<1$.  Suppose that the interval $[a,b]$ is strictly contained inside $F_{2n}([a,b],\lambda)$. So, the map $F_{2n}(\cdot,\lambda)$ has a fixed point $s_0(\lambda)$ for every $\lambda\in(\lambda_0,1)$. 
Since $s_0(\lambda)\in [a,b]$ for every $\lambda_0\in(\lambda_0,1)$ we can take a convergent subsequence $s_0(\lambda_n)$ with $\lambda_n\in(\lambda_0,1)$,  $n\geq1$, strictly increasing and converging to $1$ as $n\to\infty$. This shows that $(s_0(\lambda_n),\theta_0(\lambda_n))$ is a periodic orbit for $\Phi_{\lambda_n}$ having the same itinerary of the cylinder $\mathcal{C}$ and converging to a periodic orbit in $\mathcal{C}$. Thus, $\mathcal{C}$ is $\lambda^-$-stable and by Proposition~\ref{lambda plus minus} it is $\lambda$-stable. 
Now we show that $s_0(1)\in(a,b)$. By continuity, for every $\varepsilon>0$ sufficiently small, we can find $\hat{\lambda}_0\in(\lambda_0,1)$ such that $F_{2n}(a+\varepsilon,\lambda)<a+\varepsilon$ and $F_{2n}(b-\varepsilon,\lambda)>b-\varepsilon$ for every $\lambda\in(\hat{\lambda}_0,1)$. Therefore, $s_0(1)\in[a+\varepsilon,b-\varepsilon]$ by repeating the first part of the proof.

To prove the other direction suppose that $\mathcal{C}$ is $\lambda$-stable. As in the proof of Proposition~\ref{lambda plus minus}, let $s_0(\lambda)\in I_\mathcal{C}$ denote the base foot point corresponding to a periodic orbit $q(\lambda)$ of $\Phi_{\lambda}$ which converges to a periodic orbit in $\mathcal{C}$ as $\lambda\to1$. Since $s_0(1)\in I_\mathcal{C}$, there is $\varepsilon>0$ small such that $J_\varepsilon:=(s_0(1)-\varepsilon,s_0(1)+\varepsilon)\subset I_\mathcal{C}$ and $s_0(\lambda)\in J_\varepsilon$ for every $\lambda$ sufficiently close to $1$. Because  $F_{2n}(\cdot,\lambda)$ is affine, orientation-preserving and expanding when $\lambda<1$, we conclude that $J_\varepsilon\subsetneq F_{2n}(J_\varepsilon,\lambda)$ for every $\lambda<1$ sufficiently close to $1$. This concludes the proof.
\end{proof}


\begin{lemma}\label{suf condition}
Let $a,b\in I_\mathcal{C}$ with $a<b$. If 
$$
\frac{\partial F_{2n}}{\partial\lambda}(a,1)>0\quad\text{and}\quad\frac{\partial F_{2n}}{\partial\lambda}(b,1)<0,
$$ then the periodic cylinder $\mathcal{C}$ is $\lambda$-stable. Moreover, the $\lambda$-stable periodic orbit contained in $\mathcal{C}$ has base foot point $s\in(a,b)$. 
\end{lemma}

\begin{proof}
By continuity, there exists $\lambda_0>0$ such that for every $\lambda\in(\lambda_0,1)$ we have $F_{2n}(a,\lambda)<a$ and $F_{2n}(b,\lambda)>b$. So, the cylinder $\mathcal{C}$ is $\lambda$-stable by Proposition~\ref{lem: F}.
\end{proof}


%

Given $s\in I_\mathcal{C}$, define for each $k=0,\ldots,2n-1$, 
$$
s_k:=F_{k}(s,1)\quad\text{and}\quad \theta_k:=\lim_{\lambda\to1}\theta_k(\lambda)=\hat{\theta}_k.
$$
Recall that $t(s_k,\theta_k)$ is the oriented length of the segment joining the points $x_{i_k}(s_k)$ and $x_{i_{k+1}}(s_{k+1})$. 
\begin{lemma}\label{derivative F} For every $s\in I_\mathcal{C}$,
$$
\frac{\partial F_{2n}}{\partial\lambda}(s,1)=\sum_{i=0}^{2n-1}\frac{t(s_i,\theta_i)}{\cos\theta_{0}}\left.\frac{d\theta_{0}}{d\lambda}\right|_{\lambda=1}+\sum_{k=1}^{2n-1}(-1)^k\theta_{k}\sum_{i=k}^{2n-1}\frac{t(s_i,\theta_i)}{\cos\theta_{0}}.
$$
\end{lemma}
\begin{proof}
Let $H$ denote the map $(s,\theta,\lambda)\mapsto (R_\lambda\circ\Phi(s,\theta),\lambda)$. When instead of $\Phi$ we use the restriction maps $\Phi_{i,j}$, then we write $H_{i,j}$ in place of $H$. 
By \eqref{derivative} we have
$$
DH(s,\theta,\lambda)=-\begin{pmatrix}\frac{\cos\theta}{\cos\bar{\theta}}&\frac{t(s,\theta)}{\cos\bar{\theta}}&0\\0&\lambda&-\bar{\theta}\\0&0&-1\end{pmatrix}.
$$
Define $H_k:=H_{i_{k-1},i_{k}}\circ\cdots\circ H_{i_0,i_1}$, for $k=1,\ldots,2n$. Hence,
$$
D H_k(s_0,\theta_0,\lambda)=(-1)^{k}\begin{pmatrix}\alpha_k(\theta_0)&\gamma_k(s_0,\theta_0)&\sum_{j=1}^{k-1}(-1)^j\bar{\theta}_{j-1}\gamma_{k-j}(s_j,\theta_j)\\0&\lambda^k&\sum_{j=1}^k(-1)^j\lambda^{k-j}\bar{\theta}_{j-1}\\0&0&(-1)^k\end{pmatrix},
$$
where
\begin{align*}
\alpha_k(\theta_0)&:=\prod_{j=0}^{k-1}\frac{\cos\theta_j}{\cos\bar{\theta}_j},\\
\gamma_{k-j}(s_j,\theta_j)&:=\sum_{i=j}^{k-1}\lambda^{i-j} \frac{t(s_i,\theta_i)}{\cos\bar{\theta}_{i}}\frac{\cos\theta_{i+1}}{\cos\bar{\theta}_{i+1}}\cdots\frac{\cos\theta_{k-1}}{\cos\bar{\theta}_{k-1}},\quad j=0,\ldots,k-1,\\
(s_j,\theta_j)&:=R_\lambda\circ \Phi_{i_{j-1},i_j}(s_{j-1},\theta_{j-1}),\quad j=1,\ldots,k-1,\\
\bar{\theta}_j&:=\beta_{i_j,i_{j+1}}-\theta_j,\quad j=0,\ldots,k-1.
\end{align*}
Now define 
$$
G(s,\theta,\lambda):=\pi_1\circ H_{2n}(s,\theta,\lambda),$$
where $\pi_1$ is the projection $(s,\theta,\lambda)\mapsto s$. 
Using the formulas above we compute,
\begin{align*}
\frac{\partial G}{\partial \theta} (s_0,\theta_0,\lambda)&=\gamma_{2n}(s_0,\theta_0),\\
\frac{\partial G}{\partial \lambda} (s_0,\theta_0,\lambda)&=\sum_{k=1}^{2n-1}(-1)^k\bar{\theta}_{k-1}\gamma_{2n-k}(s_k,\theta_k),
\end{align*}
where 
$$
\gamma_{2n-k}(s_k,\theta_k)=\sum_{i=k}^{2n-1}\lambda^{i-k} \frac{t(s_i,\theta_i)}{\cos\bar{\theta}_{2n-1}}\rho_\lambda(\bar{\theta}_{i-1})\cdots\rho_\lambda(\bar{\theta}_{2n-2}),
$$
and $\rho_\lambda$ is the function defined in Proposition~\ref{prop-odd}.
Since $F_{2n}(s,\lambda)=G(s,\theta_0(\lambda),\lambda)$, we get
\begin{align*}
\frac{\partial F_{2n}}{\partial\lambda}(s,\lambda)&=\frac{\partial G}{\partial \theta} (s,\theta_0(\lambda),\lambda)\frac{d\theta_{0}}{d\lambda}+\frac{\partial G}{\partial \lambda} (s,\theta_0(\lambda),\lambda)\\
&=\gamma_{2n}(s,\theta_0(\lambda))\frac{d\theta_{0}}{d\lambda}+\sum_{k=1}^{2n-1}(-1)^k\bar{\theta}_{k-1}(\lambda)\gamma_{2n-k}(F_k(s,\lambda),\theta_k(\lambda)).
\end{align*}
Setting $\lambda=1$ we obtain,
$$
\frac{\partial F_{2n}}{\partial\lambda}(s,1)=\sum_{i=0}^{2n-1}\frac{t(s_i,\theta_i)}{\cos\theta_{0}}\left.\frac{d\theta_{0}}{d\lambda}\right|_{\lambda=1}+\sum_{k=1}^{2n-1}(-1)^k\theta_{k}\sum_{i=k}^{2n-1}\frac{t(s_i,\theta_i)}{\cos\theta_{0}}.
$$
\end{proof}

Let
$$
L_k(s):=\sum_{i=0}^{k-1}t(s_i,\theta_i).
$$
For convenience we set $L_0(s):=0$. Note that $L_{2n}(s)=L$ for every $s\in I_\mathcal{C}$ where $L$ denotes the length of the periodic cylinder $\mathcal{C}$.

%
Define,
$$
\Omega_0:=\frac{1}{4n}\sum_{k=0}^{2n-1}(-1)^{k+1} k(2n-k)\beta_{i_k,i_{k+1}}.
$$

\begin{theorem}\label{suf prop}
If there are $a,b\in I_\mathcal{C}$ with $a<b$ such that,
\begin{equation}\label{main suf condition}
\sum_{k=1}^{2n}(-1)^k\theta_{k}L_k(a)<\Omega_0L<\sum_{k=1}^{2n}(-1)^k\theta_{k}L_k(b),
\end{equation}
then the periodic cylinder $\mathcal{C}$ is $\lambda$-stable. Moreover, the $\lambda$-stable periodic orbit contained in $\mathcal{C}$ has base foot point $s\in(a,b)$.
\end{theorem}

\begin{proof}
By a simple calculation we obtain,
$$
\left.\frac{d\theta_{0}}{d\lambda}\right|_{\lambda=1}=\frac{1}{4n}\sum_{k=0}^{2n-1}(-1)^{k+1} k(2n-k)\beta_{i_k,i_{k+1}}=\Omega_0.
$$
Thus we get by Lemma~\ref{derivative F},
\begin{align*}
\frac{\partial F_{2n}}{\partial\lambda}(a,1)&=\frac{1}{\cos\theta_0}\left(L\left.\frac{d\theta_{0}}{d\lambda}\right|_{\lambda=1}+\sum_{k=1}^{2n-1}(-1)^k\theta_{k}(L-L_k(a))\right)\\
&=\frac{1}{\cos\theta_0}\left(L(\Omega_0-\theta_0)+\sum_{k=0}^{2n-1}(-1)^k\theta_{k}(L-L_k(a))\right)\\
&=\frac{1}{\cos\theta_0}\left(L(\Omega_0-\theta_0)-\sum_{k=0}^{2n-1}(-1)^k\theta_{k}L_k(a)\right)\\
&=\frac{1}{\cos\theta_0}\left(L(\Omega_0-\theta_{2n})-\sum_{k=0}^{2n-1}(-1)^k\theta_{k}L_k(a)\right)\\
&=\frac{1}{\cos\theta_0}\left(L\Omega_0-\sum_{k=1}^{2n}(-1)^k\theta_{k}L_k(a)\right),
\end{align*}
where we have used the fact that 
\begin{align*}
\sum_{k=0}^{2n-1}(-1)^k\theta_k&=\sum_{k=0}^{2n-1}(-1)^k\left(\frac{1}{2n}\sum_{j=0}^{2n-1}(-1)^{j+1}j\beta_{i_{k+j},i_{k+j+1}}\right)\\
&=\frac{1}{2n}\sum_{j=0}^{2n-1}(-1)^{j+1} j\sum_{k=0}^{2n-1}(-1)^{k}\beta_{i_{k+j},i_{k+j+1}}=0,
\end{align*}
where the indices $k$ of $\beta_{i_k,i_{k+1}}$ are taken mod $2n$ and the last equality follows from Lemma~\ref{prop-angles}. Because $\cos\theta_0>0$, we have $\frac{\partial F_{2n}}{\partial\lambda}(a,1)>0$ if and only if
$$
L\Omega_0>\sum_{k=1}^{2n}(-1)^k\theta_{k}L_k(a).
$$
The computation for $\frac{\partial F_{2n}}{\partial\lambda}(b,1)$ is analogous. The theorem now follows from Lemma~\ref{suf condition}.
\end{proof}

\begin{remark}\label{l and r base foot points}
Denote by $l$ and $r$ the left and right endpoints of the interval $I_\mathcal{C}$. Then $l$ and $r$ are respectively the base foot points of the left and right generalized diagonals forming the boundary of $\mathcal{C}$. Following the proof of Theorem~\ref{suf prop} it is easy to see that $\mathcal{C}$ is $\lambda$-stable provided condition \eqref{main suf condition} holds with $a=l$ and $b=r$.  
\end{remark}

\begin{remark}\label{not lambda stable}
The map $F_{2n}(\cdot,\lambda)$ is affine and expanding by Lemma~\ref{F}. Therefore, if $\frac{\partial F_{2n}}{\partial\lambda}(\cdot,1)$ has the same sign at $l$ and $r$, i.e.,
$$
\frac{\partial F_{2n}}{\partial\lambda}(l,1)\frac{\partial F_{2n}}{\partial\lambda}(r,1)>0,
$$
then $\mathcal{C}$ cannot be $\lambda$-stable. 
\end{remark}

\section{Applications}\label{sec:results}

In this section we collect some applications of the results proved in section~\ref{nec and suf section}. We defer the proofs of the following statements to section \ref{sec:proofs}. 

A polygon $P$ is called \textit{rational} if all its angles are rational multiples of $\pi$. 
 
 \begin{theorem}\label{thm-finite}
In any rational polygon with vertices having algebraic coordinates there are at most finitely many  $\lambda$-stable periodic cylinders.
\end{theorem}

The applicability of this theorem is strongly related to the results of \cite{C-S}, but we give self contained simple proofs.
 Theorem \ref{thm-finite} applies to regular polygons since we can write the vertices of any regular polygon as $(\cos(2\pi k/d),\sin(2\pi k/d))$, with $k=0,\ldots,d-1$. Moreover, Theorem~\ref{thm-finite} does not apply to all rational polygons since not every rational polygon has algebraic vertices. In fact, there are rational quadrilaterals having transcendental vertices, e.g., in rectangles of
height 1, most widths are not algebraic.
However, as we show in the following corollary, all rational triangles have algebraic vertices. 
 \begin{corollary}\label{corollary triangles}
In any rational triangle there are at most finitely many  $\lambda$-stable periodic cylinders.
\end{corollary}

\begin{proof}
Consider the triangle having vertices $(0,0)$, $(1,0)$ and $p_a:=(a,a\tan\alpha)$ where $a>0$ and $0<\alpha<\pi/2$. The triangle has angle $\alpha$ at $(0,0)$. Denote by $\beta$ the angle at $p_a$. Then
\begin{equation}\label{eq:cos2beta}
\cos(2\beta)=1-\frac{2\sin^2\alpha}{(a-1)^2+a^2\tan^2\alpha}.
\end{equation}
So, if both $\alpha$ and $\beta$ are rational multiple of $\pi$, then $a$ is algebraic because $\cos(2\beta)$, $\sin\alpha$ and $\tan\alpha$ are algebraic. Thus, all rational triangles have algebraic vertices and the conclusion follows from Theorem~\ref{thm-finite}.
\end{proof}
\begin{remark}
Not every triangle with algebraic vertices is rational. Indeed, let $\alpha=\pi/4$ and $a=\frac{\sqrt{n-1}+\sqrt{n}-1}{2 \left(\sqrt{n}-1\right)}$ with odd $n\geq3$ 
, which is the positive root of the polynomial 
$$
q(x)=\left(1-\frac{1}{\sqrt{n}}\right)\left((x-1)^2+x^2\right)-1.
$$
By \eqref{eq:cos2beta}, and our choice of $\alpha$ and $a$ we get $\cos(2\beta)=\frac{1}{\sqrt{n}}$. However, by \cite[Theorem 3 of Chapter 8]{A-Z}, we know that $\frac{1}{2\pi}\arccos\frac{1}{\sqrt{n}}$ is irrational. Hence, $\beta$ is not a rational multiple of $\pi$.  Thus, the triangle with algebraic vertices $(0,0)$, $(1,0)$ and $(a,a\tan\alpha)$ is not rational.
\end{remark}


Is there a polygon with an infinite number of $\lambda$-stable periodic cylinders? Having no $\lambda$-stable periodic cylinders at all? At the time of writing, we do not know the answer to these questions. 

However, we know from Theorem~\ref{thm-finite} that a dense set of polygons have at most a finite number of $\lambda$-stable periodic cylinders. The space of $d$-gons can be considered as a subset of the Euclidean space $\Rr^{2d}$, thus inheriting its topology. 

\begin{corollary}
The set of $d$-gons having at most finite number of $\lambda$-stable periodic cylinders is dense in the space of $d$-gons. 
\end{corollary}

\begin{proof}
It is well known that rational polygons are dense in the space of $d$-gons. By the proof of Corollary~\ref{corollary triangles}, we know that rational triangles have algebraic vertices, thus form a dense subset of the set of all triangles. Now given a polygon $P$ consider any triangulation of it, i.e., break $P$ into a union of triangles with non-intersecting interior and vertices belonging to the set of vertices of $P$. Such triangulation can be obtained in many ways, e.g., the ear clipping method.

Now pick a triangle $\Delta_0$ in the triangulation. We can approximate it by another triangle $\Delta_0'$ having rational angles and algebraic vertices. Next, pick a triangle $\Delta_1$ adjacent to $\Delta_0'$ and denote by $v_1$ the vertex of $\Delta_1$ which is not a vertex of $\Delta_0'$. We can move $v_1$ around and $\Delta_0'$ does not change. Using the construction in the proof of Corollary~\ref{corollary triangles} we can perturb slightly $v_1$ obtaining a rational triangle $\Delta_1'$ with algebraic vertices and adjacent to $\Delta_0'$.  Now, repeat this procedure until exhausting the list of triangles forming the triangulation of $P$. At the end, we have perturbed  the triangulation of $P$ to obtain an arbitrary close polygon $P'$ whose triangulation consists of rational triangles having algebraic vertices. Thus $P'$ is also rational and has algebraic vertices. Using Theorem~\ref{thm-finite}, this shows that in any neighbourhood of $P$ we can find $P'$ having at most finite number of $\lambda$-stable periodic cylinders. 
\end{proof}

A {\it (generalized) Fagnano\footnote{Giovanni Francesco Fagnano, was an italian mathematician that lived between the years 1715 -- 1797.} orbit} is a periodic orbit with period $p$ of the billiard in $P$ which has itinerary $i_{k+1}= i_k + m\pmod{d}$ for $k=0,\ldots,p-1$ and some integer $1\leq m< d$. Every Fagnano orbit extends to a periodic cylinder called the \textit{Fagnano cylinder}. 


\begin{proposition}[Fagnano]\label{regular}
If $P$ is a regular polygon, then every Fagnano cylinder is $\lambda$-stable. Moreover, the periodic orbit in the center of the cylinder is $\lambda$-stable.
\end{proposition} 

In the following we study the $\lambda$-stable periodic cylinders in integrable polygons, i.e., the rectangles, the equilateral triangle and the right triangles 45-45-90 and 30-60-90 degrees. 
Except for the rectangle we determine exactly the number of $\lambda$-stable periodic cylinders and show that in all cases there exists at least one and at most a finite number of $\lambda$-stable periodic cylinders. For integrable triangles the finiteness of $\lambda$-stable periodic cylinders follows from Corollary~\ref{corollary triangles}.

\begin{proposition}[Rectangle]\label{prop:rectangle} Every rectangle has two ping-pong cylinders, they are $\lambda$-stable.
These are the only $\lambda$-stable periodic cylinders unless the aspect ratio $w$  of the rectangle
satisfies 
\begin{equation}\label{aspect ratio formula}
w = \frac{p}{q} \cot \left ( \frac{\pi p}{2(p + q)} \right )
\end{equation} 
with $p$ and $q$ positive integers. 
In addition to the ping-pong cylinders, when $w$ satisfies \eqref{aspect ratio formula}, there are at most two more $\lambda$-stable periodic cylinders. The slope of the cylinders is $\pm p/(qw)$. 
\end{proposition}


This proposition generalizes immediately to polygons which are tiled by translated copies of a single rectangle. In the case of the square we get the following result.

\begin{corollary}[Square]\label{prop-square}
The only $\lambda$-stable periodic cylinders in the square are the two ping-pond cylinders and the two Fagnano cylinders (having period four). 
\end{corollary}

The following result shows that $\lambda$-stability and the notation of stability introduced by Galperin, Stepin and Vorobets are unrelated. 

\begin{corollary}\label{GSV-stability} 
There are  cylinders of periodic orbits which are 
\begin{enumerate}
\item stable and $\lambda$-stable
\item stable and not $\lambda$-stable,
\item unstable and $\lambda$-stable.
\item unstable and $\lambda$-unstable.
\end{enumerate}
\end{corollary}

\begin{proof}
Every odd length periodic orbit is stable (\cite[p.28]{GSV}) and $\lambda$-stable.
Every periodic orbit in the square is unstable (\cite[Theorem 1]{GSV}) while only the ping-pong orbits 
and two period four orbit are $\lambda$-stable (Corollary~\ref{prop-square}).
There are infinitely many stable periodic orbits in the equilateral triangle (\cite[Assertion 9 and Theorem 3]{GSV})
but only finitely many $\lambda$-stable orbits (Corollary~\ref{corollary triangles}).

\end{proof}

Regarding the first item of the previous corollary, the following natural question arises: is there an even length periodic orbit which is both stable and $\lambda$-stable?

Now we turn to triangles. The following result concerns the equilateral triangle. A cylinder of periodic orbits is called \textit{perpendicular} if the angle of departure of the periodic orbits is zero. Such cylinders of periodic orbits must have even period.

\begin{proposition}[Equilateral triangle]\label{prop-equilateral}
In the equilateral triangle only the two Fagnano cylinders and the two perpendicular periodic cylinders of period four are $\lambda$-stable.
\end{proposition}

Using the same ideas as in the proof of Proposition~\ref{prop-equilateral} we obtain the following result.

\begin{proposition}[Regular hexagon]\label{prop:hexagon}
In the regular hexagon only the two Fagnano periodic cylinders of period six, the four Fagnano periodic cylinders of period three and the three ping-pong cylinders are $\lambda$-stable.
\end{proposition}

Regarding integrable right triangles we have the following results.

\begin{proposition}[Right triangle with 30-60-90 degrees] \label{306090}
In the right triangle with 30-60-90 degrees only the following periodic cylinders are $\lambda$-stable:
\begin{enumerate}
\item the two periodic cylinders of period six which are perpendicular to the shorter leg and the diagonal of the triangle,
\item the two periodic cylinders of period six which are perpendicular to the diagonal of the triangle,
\item the two periodic cylinders of period ten which are perpendicular to the longer leg of the triangle.
\end{enumerate}
\end{proposition}

For the isosceles right triangle we have,

\begin{proposition}[Right triangle with 45-45-90 degrees] \label{454590}
In the right triangle with 45-45-90 degrees only the two periodic cylinders of period four which are perpendicular to the legs of the triangle and the two periodic cylinders of period six which are perpendicular to the diagonal of the triangle are $\lambda$-stable.
\end{proposition}

\section{Proofs}\label{sec:proofs}

In this section we collect the proofs of the results stated in section~\ref{sec:results}.

\subsection{Proof of Theorem~\ref{thm-finite}}

Let $P$ be a rational polygon whose vertices have algebraic coordinates. Any rational polygon has at most finitely many odd length periodic orbits (\cite[p.30]{GSV}). These are $\lambda$-stable by Proposition~\ref{prop-odd}. Now we turn to even length periodic cylinders. Let $V=\{v_i\}$ denote the set of vectors obtained by the action of the finite group of reflections $G(P)$ on the set of vertices of $P$. Clearly, $V$ is a finite set
(orientation is unimportant) and the coordinates of each $v_i\in V$ are algebraic. 
Consider a $\lambda$-stable periodic cylinder of even length $2n$, and let 
$v$ be a generalized diagonal on the boundary of this cylinder. Any generalized diagonal can be represented as a sum of the $v_i$, i.e., 
there exists $j_i \in \mathbb{Z}$ such that 
$$
v = \sum j_i v_i.
$$
Proposition \ref{necessary condition} tells us that the slope of $v$ is 
$$
\tan\left(\frac{1}{2n} \sum_{k=0}^{2n-1} (-1)^{k+1}k\beta_{i_k,i_{k+1}}\right).
$$ 
Since $P$ is rational all the angles
$\beta_{i,j}$ are rational multiples of $\pi$, thus the slope of $v$ equals $\tan(p \pi/q)$ for some co-prime integers $p$ and $q> 0$. Because all $v_i$ have algebraic coordinates, there exists an integer $N=N(P)>0$ such that the algebraic degree of the slope of $v$ is less than $N$. However, by Niven's theorem \cite[Theorem 3.11]{Niven}, the algebraic degree of $\tan(p \pi/q)$ is $\varphi(q)/2$ if $q$ is a multiple of $4$ and $\varphi(q)$ otherwise, where $\varphi$ denotes the Euler's totient function. Thus $\varphi(q)\leq 2 N$ in either case. This implies that $q\in \varphi^{-1}(\{1,\ldots,2N\})$, which is a finite set. Therefore
there are only finitely many such slopes, which correspond to a finite collection of directions $\{\theta_i: i \in I\}$ in the billiard.
Since our polygon is rational, the billiard orbits starting in a given direction $\theta$ can only take a finite number of directions,
call this set of directions $[\theta]$.  Thus the possible directions of generalized diagonals on the boundary of a $\lambda$-stable
cylinder are in the finite set $\{\theta: \theta \in [\theta_i] \text{ for some } i \in I\}$. Since each generalized diagonal starts at a corner of the polygon, i.e., in a finite
collection of points, there are only finitely many such generalized diagonals, 
 thus only finitely possible
many $\lambda$-stable cylinders in $P$.\hfill $\Box$

\subsection{Proof of Proposition~\ref{regular}}
Let $P$ be a regular $d$-gon. Label the edges of $P$ anticlockwise from 1 to $d$. If the period $p$ of the cylinder is odd, then the cylinder is $\lambda$-stable by Proposition~\ref{prop-odd}. Otherwise suppose that $p$ is even, say $p=2n$. Also suppose that the Fagnano cylinder has itinerary $\{i_k\}_{k=0}^{2n-1}$ with $i_k=1+km\pmod{d}$ for some integer $1\leq m\leq \lfloor d/2\rfloor$. 
In the case $d$ is even, $m=d/2$ gives a ping-pong cylinder which we already know that is $\lambda$-stable. So we assume that $m<d/2$. 
Then for every $k=0,\ldots,2n-1$,
$$
\beta_{i_k,i_{k+1}}=\pi\left(1-\frac{2m}{d}\right) \quad \text{and}\quad \theta_k=\frac\pi2\left(1-\frac{2m}{d}\right),
$$
where the indices $k$ of $\beta_{i_k,i_{k+1}}$ are taken mod $2n$. Therefore, 
\begin{align*}
\frac{1}{2n}\sum_{k=0}^{2n-1}(-1)^{k+1}k\beta_{i_k,i_{k+1}}&=\frac{\pi}{2n}\left(1-\frac{2m}{d}\right)\sum_{k=0}^{2n-1}(-1)^{k+1}k\\
&=\frac\pi2\left(1-\frac{2m}{d}\right).
\end{align*}
So the angle of departure of the cylinder satisfies the formula in Proposition~\ref{necessary condition}. Now we verify the condition in Theorem~\ref{suf prop} (see also Remark~\ref{l and r base foot points}). First note that  
\begin{align*}
\Omega_0&=\frac{1}{4n}\sum_{k=0}^{2n-1}(-1)^{k+1} k(2n-k)\beta_{i_k,i_{k+1}}\\
&=\frac{\pi}{4n}\left(1-\frac{2m}{d}\right)\sum_{k=0}^{2n-1}(-1)^{k+1} k(2n-k)\\
&=\frac{\pi}{4}\left(1-\frac{2m}{d}\right).
\end{align*}
Hence,
$$
\Omega_0 L =\frac{L\pi}{4}\left(1-\frac{2m}{d}\right),
$$
where $L$ denotes the length of the cylinder. Taking into account that, 
$$
L_{k+1}(l)-L_{k}(l)=L_1(r)\quad\text{and}\quad L_{k+1}(r)-L_{k}(r)=L_1(l)
$$
for $k=1,3,5,\ldots,2n-1$, we get
\begin{align*}
\sum_{k=1}^{2n}(-1)^k\theta_{k}L_k(l)&=\frac\pi2\left(1-\frac{2m}{d}\right)\sum_{k=1}^{2n}(-1)^k L_k(l)\\&=\frac\pi2\left(1-\frac{2m}{d}\right)nL_1(r).
\end{align*}
Similarly,
$$
\sum_{k=1}^{2n}(-1)^k\theta_{k}L_k(r)=\frac\pi2\left(1-\frac{2m}{d}\right)nL_1(l).
$$
Since $2n L_1(r)<L<2n L_1(l)$, the hypothesis of Theorem~\ref{suf prop} holds, and so the cylinder is $\lambda$-stable. To show the last claim let $m\in I_\mathcal{C}$ denote the midpoint of the interval $I_\mathcal{C}$ and note that $L_{k+1}(m)-L_k(m)=L/(2n)$ for $k=0,\ldots,2n-1$. Thus,
\begin{align*}
\sum_{k=1}^{2n}(-1)^k\theta_{k}L_k(m)&=\frac\pi2\left(1-\frac{2m}{d}\right)\sum_{k=1}^{2n}(-1)^k L_k(m)\\
&=\frac{L\pi}{4}\left(1-\frac{2m}{d}\right)=\Omega_0 L.
\end{align*}
Now we can take arbitrarily small $\epsilon>0$ and set $a=m-\epsilon$ and $b=m+\epsilon$ so that the hypothesis of Theorem
\ref{suf prop} holds for this choice of $a$ and $b$. Hence, the periodic orbit in the center of the cylinder is $\lambda$-stable. \hfill $\Box$

\subsection{Proof of Proposition~\ref{prop:rectangle}}

Each rectangle has two ping-pong cylinders, they are stable.

By rotating and rescaling we can  suppose that the rectangle has vertical and horizontal sides width $w$ and height $1$.
Then we have a periodic orbit of period $2(|p|+|q|)$ if and only if  $p$ and $q$ are relatively prime and the 
slope of the orbit is $p/(qw)$. 
We consider the periodic cylinder starting on a vertical side with right boundary the origin.  See Figure~\ref{2/5}, drawn in the case of aspect ratio 1.

%

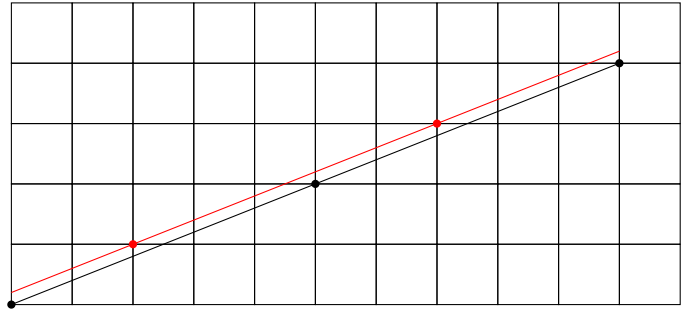
\begin{figure}[h]
\begin{tikzpicture}[scale=0.8]
\foreach \i in {0,...,10}    
\foreach \j in {0,...,4}
 \draw (\i,\j) rectangle (\i+1,\j+1);
 \draw[thick,fill] (0,0)  circle (0.05);
   \draw[thick,fill]  (5,2)  circle (0.05);
      \draw[thick,fill]  (10,4)  circle (0.05);
            \draw[thick,fill,red]  (2,1)  circle (0.05);
            \draw[thick,fill,red]  (7,3)  circle (0.05);
 \draw[] (0,0) -- (10,4);
  \draw[red] (0,0.2) -- (10,4.2);
 
\end{tikzpicture}
\caption{Slope $2/5$.  {\footnotesize $( \beta_{i_k,i_{k+1}})_{k =0}^{13} = (0, 0, \minus \frac{\pi}2, \minus \frac{\pi}2, 0, \frac{\pi}2, \frac{\pi}2,0, 0, \frac{\pi}2,  \frac{\pi}2, 0, \minus\frac{\pi}2, \minus\frac{\pi}2)$} }\label{2/5}
 \end{figure}

Consider the alternating sum of Proposition~\ref{necessary condition}
$$
\frac{1}{2(p + q)}\sum_{k=0}^{2(p + q)-1}(-1)^{k+1}k\beta_{i_k,i_{k+1}}.
$$

We suppose that $p$ and $q$ are strictly positive, and that $p \le q$. 
We claim that, other than the vertical and horizontal directions, the other cases reduce to this case.
First of all we can assume $p$ are $q$ are strictly positive using the symmetries of the rectangle.
Furthermore note that  there is a $\lambda$-stable periodic orbit with slope $p/qw$ in the rectangle with aspect ratio
$w$ if and only if there there is a $\lambda$-stable periodic orbit with slope $q/pw^{-1}$ in the rectangle with aspect ratio
$w^{-1}$.

%

First consider the case $p=q$, then
since $p$ and $q$ are relatively prime we have $p=q=1$ and all the $\beta_{i_k,i_{k+1}}$ equal $- \frac{\pi}{2}$, thus 
$$
\frac{1}{2(p + q)}\sum_{k=0}^{2(p + q)-1}(-1)^{k+1}k\beta_{i_k,i_{k+1}}=
\frac{-\pi}{8}\sum_{k=0}^{3}(-1)^{k+1}k = -\frac{\pi}{4}.
$$
Now consider the case $p<q$. Note that the terms coming from hitting opposite sides do not contribute to the sum since 
for any $i$ we have
$\beta_{i,i+2} = 0$. The contribution comes from the $\beta_{i_k,i_{k+1}}$ which correspond to a collision between adjacent sides. In that case $\beta_{i_k,i_{k+1}}=\pm\pi/2$.

We can always decompose the sequence of angles $(\beta_{i_k,i_{k+1}})_{k=0}^{2(p+q)-1}$ in blocks of the form 
$$
\left(\beta_{i_r,i_{r+1}},\ldots,\beta_{i_{s},i_{s+1}},\beta_{i_{s+1},i_{s+2}},\ldots,\beta_{i_{t},i_{t+1}}\right)
$$
with $n:=s-r+1$ angles $\beta_{i_r,i_{r+1}},\ldots,\beta_{i_{s},i_{s+1}}$ corresponding to consecutive collisions between opposite sides ($\beta_{i,i+2} = 0$), and $m:=t-s$ angles $\beta_{i_{s+1},i_{s+2}},\ldots,\beta_{i_{t},i_{t+1}}$ corresponding to consecutive collisions between adjacent sides ($\beta_{i,i+1} = -\beta_{i+1,i}=\pi/2$). Note that $m$ and $n$ are non-negative integers (might be zero) but $m$ is always even. Let us suppose that there are $N\geq1$ such blocks. Each block, say the $j$-th block, with $n_j:=s_j-r_j+1$ and $m_j:=t_j-s_j$ corresponding to the number of collisions between opposite and adjacent sides, respectively. In the example depicted in Figure~\ref{2/5} there are 4 blocks,
$$
\left\{(0,0,\minus \,\frac{\pi}2, \minus\,\frac{\pi}2),(0, \frac{\pi}2, \frac{\pi}2),(0,0,\frac{\pi}2,\frac{\pi}2),(0, 
 \minus\,\frac{\pi}2, \minus\,\frac{\pi}2)\right\}.
$$
Note that
\begin{equation}\label{eq indices}
r_j=n_1+\cdots+n_{j-1}+m_1+\cdots+m_{j-1}\quad\text{and}\quad s_j+1=r_j+n_j.
\end{equation}
Thus, the contribution of each block is
\begin{align*}
\sum_{k=r_j}^{t_j}(-1)^{k+1}k\beta_{i_k,i_{k+1}}&=\sum_{k=r_j}^{s_j}(-1)^{k+1}k\beta_{i_k,i_{k+1}}+\sum_{k=s_j+1}^{t_j}(-1)^{k+1}k\beta_{i_k,i_{k+1}}\\
&=\sum_{k=s_j+1}^{t_j}(-1)^{k+1}k\beta_{i_k,i_{k+1}}\\
&=\beta_{i_{s_j+1},i_{s_j+2}}\sum_{k=s_j+1}^{t_j}(-1)^{k+1}k\\
&=\beta_{i_{s_j+1},i_{s_j+2}}(-1)^{s_j}\sum_{k=0}^{m_j-1}(-1)^k(s_j+1+k).
\end{align*}
Clearly,
$$
\sum_{k=0}^{m_j-1}(-1)^k(s_j+1+k)=-\frac{m_j}{2}.
$$
Moreover, $\beta_{i_{s_j+1},i_{s_j+2}}=\pm\frac\pi2$. In fact, since $\beta_{i_{s_j+1},i_{s_j+2}}=(-1)^{n_j}\beta_{i_{s_{j-1}+1},i_{s_{j-1}+2}}$ for every $2\leq j\leq N$, and $\beta_{i_{s_1+1},i_{s_1+2}}=(-1)^{n_1+1}\frac\pi2$ we have,
$$
\beta_{i_{s_j+1},i_{s_j+2}}=(-1)^{n_1+\cdots+n_j+1}\frac{\pi}{2}.
$$ 
Hence,
$$
\sum_{k=r_j}^{t_j}(-1)^{k+1}k\beta_{i_k,i_{k+1}}=\frac{m_j\pi}{4} (-1)^{n_1+\cdots+n_j+s_j}.
$$
By \eqref{eq indices} we have,
$$
n_1+\cdots+n_j+s_j+1=2(n_1+\cdots+n_j)+m_1+\cdots+m_{j-1},
$$
which is an even number since the $m_i$ are even. Thus,
$$
\sum_{k=r_j}^{t_j}(-1)^{k+1}k\beta_{i_k,i_{k+1}}=-\frac{m_j\pi}{4}.
$$
Collecting all the $N$ blocks we obtain,
$$
\frac{1}{2(p + q)}\sum_{k=0}^{2(p + q)-1}(-1)^{k+1}k\beta_{i_k,i_{k+1}}=-\frac{\pi}{8(p + q)}\sum_{j=1}^Nm_j.
$$
Every time the cylinder crosses a horizontal line it corresponds to two consecutive collisions between adjacent sides. Since there are $2p$ such crossings we have
$$
\sum_{j=1}^Nm_j = 4p.
$$
This implies that the periodic cylinder with slope $p/(qw)$ is $\lambda$-stable in the rectangle with aspect ratio $w$ only if $p$ and $q$ satisfy the equation
$$
\frac{p}{qw}=\tan\left(\frac{\pi p}{2(p + q)}\right).
$$
This concludes the proof.\hfill $\Box$

\subsection{Proof of Corollary~\ref{prop-square}}
Proposition \ref{prop:rectangle} shows that
a periodic cylinder with slope $p/q$ is $\lambda$-stable in the square only if 
$$
\frac{p}{q}  = \tan \left ( \frac{\pi p}{2(|p|+|q|)} \right ).
$$
It is known that the only rational values of the tangent function at rational multiples of $\pi$
are $0$ and $\pm 1$ \cite{MR2519490}.  In particular the only solutions are $p =0$ and $|p| =  |q| =  1$.  These solutions correspond to the slope $0$ of period 2 ping-pong cylinder and the slope $\pm 1$ of period four cylinders (Fagnano cylinders). The two Fagnano cylinders are  $\lambda$-stable by Proposition~\ref{regular}. \hfill $\Box$

\subsection{Proof of Proposition~\ref{prop-equilateral}}
The two Fagnano orbits are the only orbits primitive orbits of odd period \cite{Baxter-Umble} (and thus $\lambda$-stable by Proposition \ref{prop-odd}).

Now consider even length periodic orbits, suppose that one side of the triangle is horizontal and of length 1.
Consider the vectors $(1,0)$ and $(\cos \pi/3, \sin \pi/3)$ as a basis of the plane.
Tessellate  the plane with equilateral triangles and express the vertices of these triangles in this basis.
Then the point $(p,q)$ in this basis is the point $(p+q \cos(\pi/3), q \sin(\pi/3)) = (p + q/2, q\sqrt{3}/2)$ in rectangular coordinates.
Thus the slope of a periodic orbit is of the form 
$ \frac{q\sqrt{3}}{2p+q}$.

Consider a $\lambda$-stable periodic cylinder, it must hit the horizontal side of the triangle.  Consider the angle $\theta$ which the 
cylinder makes with this side at the time of (some) collision.
Proposition \ref{necessary condition} tells us that $\theta$ is a rational multiple of $\pi$.
We conclude that
any $\lambda$-stable periodic cylinder must have a slope equal to the tangent of an angle which is a rational multiple of $\pi$. 
But it is known that the only rational multiples of $\pi$ for which the tangent  is a rational multiple of $\sqrt{3}$ are $0,  \pm \pi/6, \pm \pi/3, \pm 2\pi/3,  \pm 5\pi/6$ \cite{Calcut}.
Since $\theta \in (-\pi/2,\pi/2)$ we have  $\theta \in  \{0,  \pm \pi/6, \pm \pi/3\}$

The only periodic cylinder with angles in the set $\{\pm \pi/6\}$ are  the two Fagnano cylinders, which we already know are $\lambda$-stable.
 Each of the three other   possibilities $\{0, \pm \pi/3\}$  corresponds to a  twice periodic cylinder $\mathcal{C}$ and $\mathcal{C}^{-1}$ of period 4. Now, using Theorem~\ref{suf prop} (see also Remark~\ref{l and r base foot points}), we show that $\mathcal{C}$ is $\lambda$-stable (see Figure~\ref{fig:equilateral}). 

\begin{figure}[h]
	\centering
	\includegraphics[width=2.5in]{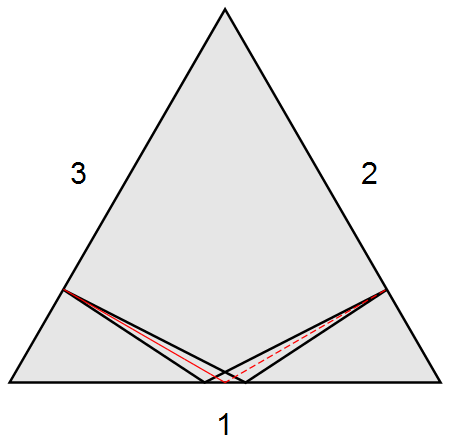}
	\caption{The dashed line represents the $\lambda$-stable periodic orbit with itinerary $\{1,2,1,3\}$ and the solid line the  associated periodic orbit for $\Phi_\lambda$ with $\lambda=0.9$.}
	\label{fig:equilateral}
\end{figure}

Label the sides of the equilateral triangle anticlockwise from 1 to 3 starting on the horizontal edge. Suppose that $\mathcal{C}$ has itinerary $\{1,2,1,3\}$. Note that $\beta_{1,2}=\beta_{3,1}=\pi/3$ and $\beta_{2,1}=\beta_{1,3}=-\pi/3$. Therefore, 
\begin{align*}
\theta_0&=\frac{1}{4}\sum_{k=0}^{3}(-1)^{k+1} k\beta_{i_k,i_{k+1}}=\frac\pi3,\\
\Omega_0&=\frac{1}{8}\sum_{k=0}^{3}(-1)^{k+1} k(4-k)\beta_{i_k,i_{k+1}}=\frac\pi6.
\end{align*}
The angle sequence $\{\theta_k\}_{k=0}^3$ of the periodic cylinder is $\{\pi/3,0,-\pi/3,0\}$. Moreover, the lengths $L_k(s)$ evaluated at the left and right generalized diagonals forming the boundary of the cylinder are
\begin{align*}
L_1(l)&=\frac{\sqrt{3}}{2},\quad L_2(l)=L_3(l)=L_4(l)=\sqrt{3},\\
L_1(r)&=L_2(r)=0,\quad L_3(r)=\frac{\sqrt{3}}{2},\quad L_4(r)=\sqrt{3}.
\end{align*}
Using these numbers we compute,
$$
\sum_{k=1}^{4}(-1)^k\theta_{k}L_k(l)=0\quad\text{and}\quad  \sum_{k=1}^{4}(-1)^k\theta_{k}L_k(r)=\frac{\pi}{\sqrt{3}}.
$$
Since $L=\sqrt{3}$ and $0<\frac{\pi}{2\sqrt{3}}<\frac{\pi}{\sqrt{3}}$, we see that the hypothesis of Theorem~\ref{suf prop} is satisfied. Thus $\mathcal{C}$ is $\lambda$-stable and by Lemma~\ref{Cminus}, $\mathcal{C}^{-1}$ is also $\lambda$-stable.
%
\hfill $\Box$

\subsection{Proof of Proposition~\ref{prop:hexagon}}
Consider even length periodic cylinders and suppose that one side of the hexagon is horizontal as in the proof of Proposition~\ref{prop-equilateral}. 
We can apply a hexagonal symmetry to assume that this cylinder hits the horizontal side.
The same calculation as in the equilateral triangle case shows that the possible angles of periodic orbits starting on this side belong to the set
 $\{0, \pm \pi/6, \pm \pi/3\}$.
There is one cylinder with angle  $0$, it is a ping-pong cylinder, thus $\lambda$-stable. There are 2 period 3 Fagnano cylinders (with angles  $\pm \pi/6)$ and two  period 6 Fagnano cylinders with angles $\pm \pi/3$,
they are $\lambda$-stable by Proposition~\ref{regular}.
\hfill$\Box$


\subsection{Proof of Proposition~\ref{306090}}
Similar to Proposition~\ref{prop-equilateral}, we consider the vectors  $(1, \tan\pi/3)$
and $(1,0)$ as generators of the tessellation of plane by right angles of 30-60-90 degrees. 
We can always consider that  periodic orbits start on the horizontal edge.
We again see that the slopes of periodic cylinders are rational multiples of $\sqrt{3}$.  Thus again the possible angles belong to the set
 $\{0, \pm \pi/3, \pm \pi/6, \pm \pi/2\} \cap (-\pi,2,\pi/2) = \{0, \pm \pi/3, \pm \pi/6\}$.

According to this set of angles we have the following possibilities:

\begin{enumerate}
\item the angle $0 $ corresponds to the period 6 periodic cylinder which is the perpendicular cylinder to the short side and the diagonal  of the triangle,
\item the angles $\pm \pi/6$ correspond to the period 10 cylinder which is perpendicular to the longer leg of the triangle,
\item the angles $\pm \pi/3$ correspond to the period 6 cylinder which is perpendicular to the diagonal of the triangle.
\end{enumerate} 
 
 As in the proof of Proposition~\ref{prop-equilateral}, we show using Theorem~\ref{suf condition} that all these perpendicular cylinders are $\lambda$-stable.  Label the sides of the triangle anticlockwise from 1 to 3 starting on the horizontal edge. Note that
 $$
 \beta_{1,2}=-\beta_{2,1}=\frac\pi2,\quad\beta_{3,1}=-\beta_{1,3}=\frac\pi3\quad\text{and}\quad  \beta_{2,3}=-\beta_{3,2}=\frac\pi6.
 $$

Let us consider the three cases separately:
\begin{enumerate}
\item 
Suppose that $\mathcal{C}$ is the perpendicular periodic cylinder with itinerary $\{1,3,2,3,2,3\}$ (see Figure~\ref{fig:306090_6b}). Then,
\begin{align*}
\theta_0&=\frac{1}{6}\sum_{k=0}^{5}(-1)^{k+1} k\beta_{i_k,i_{k+1}}=0,\\
\Omega_0&=\frac{1}{12}\sum_{k=0}^{5}(-1)^{k+1} k(6-k)\beta_{i_k,i_{k+1}}=-\frac{5\pi}{18}.
\end{align*}

\begin{figure}[h]
	
	\centering
	\includegraphics[width=2in]{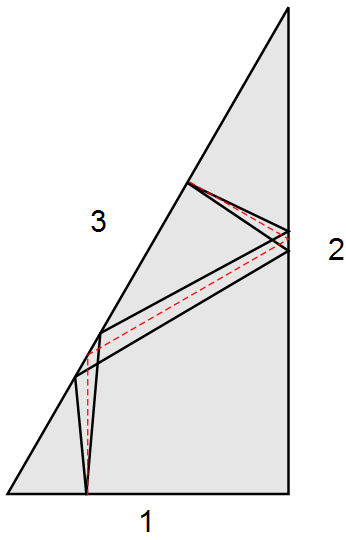}
	\caption{The dashed line represents the $\lambda$-stable periodic orbit with itinerary $\{1,3,2,3,2,3\}$ and the solid line the  associated periodic orbit for $\Phi_\lambda$ with $\lambda=0.9$.}
	\label{fig:306090_6b}
\end{figure}

The angle sequence $\{\theta_k\}_{k=0}^5$ of the periodic cylinder is 
$$
\left\{0,-\frac\pi3,\frac\pi6,0,-\frac\pi6,\frac\pi3\right\}.
$$ 
Moreover, the lengths $L_k(s)$ evaluated at the left and right generalized diagonals forming the boundary of the cylinder are
\begin{align*}
\left\{L_k(l)\right\}_{k=1}^6&=\left\{0,\frac{2}{\sqrt{3}},\sqrt{3},\frac{4}{\sqrt{3}},2\sqrt{3},2\sqrt{3}\right\},\\
\{L_k(r)\}_{k=1}^6&=\left\{\sqrt{3},\sqrt{3},\sqrt{3},\sqrt{3},\sqrt{3},2\sqrt{3}\right\}.
\end{align*}
Using these numbers we compute,
$$
\sum_{k=1}^{6}(-1)^k\theta_{k}L_k(l)=-\frac{7\pi}{3\sqrt{3}}\quad\text{and}\quad  \sum_{k=1}^{6}(-1)^k\theta_{k}L_k(r)=0.
$$
Since $L=2\sqrt{3}$ and $-\frac{7\pi}{3\sqrt{3}}<-\frac{5\pi}{3\sqrt{3}}<0$, we see that the hypothesis of Theorem~\ref{suf prop} is satisfied. Thus $\mathcal{C}$ is $\lambda$-stable and by Lemma~\ref{Cminus}, $\mathcal{C}^{-1}$ is also $\lambda$-stable.
 
\item Now suppose that $\mathcal{C}$ is the perpendicular periodic cylinder with itinerary $\{1, 2, 3, 2, 3, 2, 1, 3, 2, 3\}$ (see Figure~\ref{fig:306090_6c}). Then,
\begin{align*}
\theta_0&=\frac{1}{10}\sum_{k=0}^{9}(-1)^{k+1} k\beta_{i_k,i_{k+1}}=\frac{\pi}{6},\\
\Omega_0&=\frac{1}{20}\sum_{k=0}^{9}(-1)^{k+1} k(10-k)\beta_{i_k,i_{k+1}}=\frac{\pi}{5}.
\end{align*}

\begin{figure}[h]
	
	\centering
	\includegraphics[width=2in]{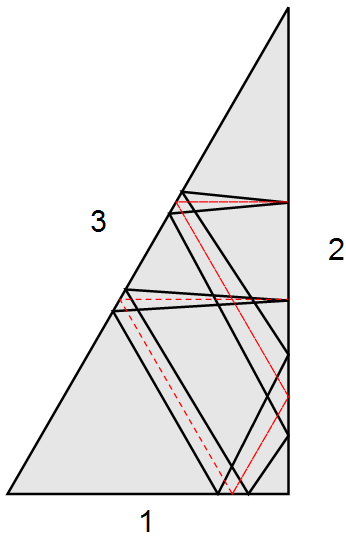}
	\caption{The dashed line represents the $\lambda$-stable periodic orbit with itinerary $\{1, 2, 3, 2, 3, 2, 1, 3, 2, 3\}$ and the solid line the  associated periodic orbit for $\Phi_\lambda$ with $\lambda=0.9$.}
	\label{fig:306090_6c}
\end{figure}

 The angle sequence $\{\theta_k\}_{k=0}^9$ of the periodic cylinder is 
$$
\left\{\frac{\pi }{6},\frac{\pi }{3},-\frac{\pi }{6},0,\frac{\pi }{6},-\frac{\pi }{3},-\frac{\pi }{6},-\frac{\pi }{6},0,\frac{\pi }{6}\right\}
$$ 
Moreover, the lengths $L_k(s)$ evaluated at the left and right generalized diagonals forming the boundary of the cylinder are
\begin{align*}
\{L_k(l)\}_{k=1}^{10}&=\{2,2,2,2,2,4,4,5,6,6\},\\
\{L_k(r)\}_{k=1}^{10}&=\left\{0,1,\frac{3}{2},2,3,3,4,\frac{9}{2},5,6\right\}.
\end{align*}
 Using these numbers we compute,
$$
\sum_{k=1}^{10}(-1)^k\theta_{k}L_k(l)=0\quad\text{and}\quad  \sum_{k=1}^{6}(-1)^k\theta_{k}L_k(r)=\frac{3\pi}{2}.
$$
Since $L=6$ and $0<\frac{6\pi}{5}<\frac{3\pi}{2}$, we see that the hypothesis of Theorem~\ref{suf prop} is satisfied. Thus $\mathcal{C}$ is $\lambda$-stable and by Lemma~\ref{Cminus}, $\mathcal{C}^{-1}$ is also $\lambda$-stable.

\item Finally, suppose that $\mathcal{C}$ is the perpendicular periodic cylinder with itinerary $\{1, 2, 3, 2, 1, 3\}$ (see Figure~\ref{fig:306090_6}). Then,
\begin{align*}
\theta_0&=\frac{1}{6}\sum_{k=0}^{5}(-1)^{k+1} k\beta_{i_k,i_{k+1}}=\frac{\pi}{3},\\
\Omega_0&=\frac{1}{12}\sum_{k=0}^{5}(-1)^{k+1} k(6-k)\beta_{i_k,i_{k+1}}=\frac{\pi}{6}.
\end{align*}

\begin{figure}[h]
	
	\centering
	\includegraphics[width=2in]{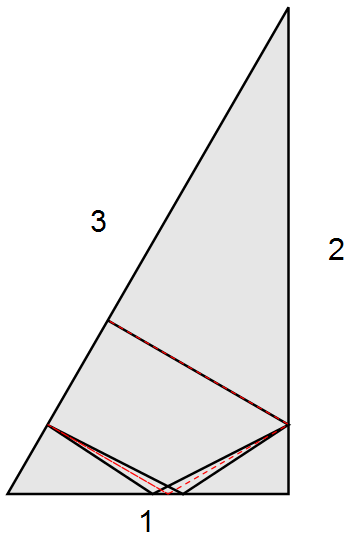}
	\caption{The dashed line represents the $\lambda$-stable periodic orbit with itinerary $\{1, 2, 3, 2, 1, 3\}$ and the solid line the  associated periodic orbit for $\Phi_\lambda$ with $\lambda=0.9$.}
	\label{fig:306090_6}
\end{figure}

The angle sequence $\{\theta_k\}_{k=0}^5$ of the periodic cylinder is 
$$
\left\{\frac{\pi }{3},\frac{\pi }{6},0,-\frac{\pi }{6},-\frac{\pi }{3},0,\frac{\pi }{3}\right\}
$$ 
Moreover, the lengths $L_k(s)$ evaluated at the left and right generalized diagonals forming the boundary of the cylinder are
\begin{align*}
\{L_k(l)\}_{k=1}^6&=\left\{\frac{2}{\sqrt{3}},\sqrt{3},\frac{4}{\sqrt{3}},2 \sqrt{3},2 \sqrt{3},2 \sqrt{3}\right\},\\
\{L_k(r)\}_{k=1}^6&=\left\{0,\frac{\sqrt{3}}{2},\sqrt{3},\sqrt{3},\frac{3 \sqrt{3}}{2},2 \sqrt{3}\right\}.
\end{align*}
 Using these numbers we compute,
$$
\sum_{k=1}^{10}(-1)^k\theta_{k}L_k(l)=\frac{\pi }{3 \sqrt{3}}\quad\text{and}\quad  \sum_{k=1}^{6}(-1)^k\theta_{k}L_k(r)=\frac{\sqrt{3} \pi }{2}.
$$
Since $L=2\sqrt{3}$ and $\frac{\pi }{3 \sqrt{3}}<\frac{\pi}{\sqrt{3}}<\frac{\sqrt{3} \pi }{2}$, we see that the hypothesis of Theorem~\ref{suf prop} is satisfied. Thus $\mathcal{C}$ is $\lambda$-stable and by Lemma~\ref{Cminus}, $\mathcal{C}^{-1}$ is also $\lambda$-stable.

\end{enumerate}

\hfill$\Box$

\subsection{Proof of Proposition~\ref{454590}}
 Suppose the triangle unfolds to a square with vertical and horizontal sides.  By unfolding we can easily see that an orbit  is periodic if and only if its  slope
is rational and that the combinatorial period of a periodic orbit is always even.

We can always consider period orbits starting on the horizontal edge.
As it the case of the square (Proposition~\ref{prop:rectangle}) the only  possible solutions of the necessary condition for $\lambda$-stable periodic cylinders of Proposition \ref{necessary condition}  are the vertical, horizontal directions as well as $\pm \pi/4$. These angles correspond to the following periodic cylinders:

\begin{enumerate}
\item two periodic cylinders of period 4 which are perpendicular to the legs of the triangle,
\item two periodic cylinders of period 6 which are perpendicular to the diagonal of the triangle.
\end{enumerate}

\begin{figure}
    \centering
    \begin{subfigure}[b]{0.4\textwidth}
        \includegraphics[width=\textwidth]{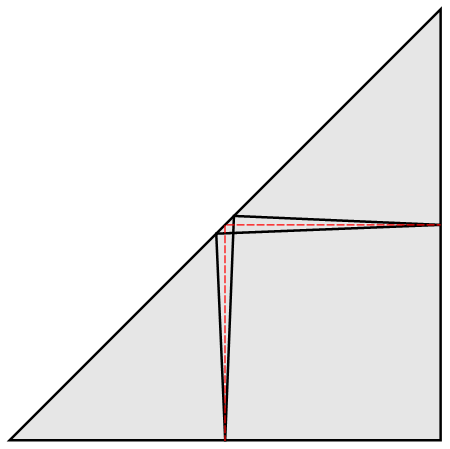}
        \caption{Period 4}
    \end{subfigure}
    $\qquad$
    \begin{subfigure}[b]{0.4\textwidth}
        \includegraphics[width=\textwidth]{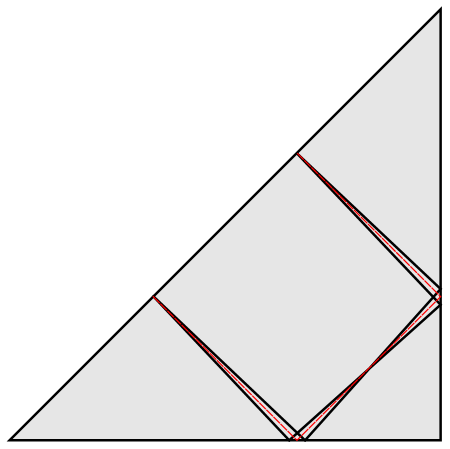}
        \caption{Period 6}
    \end{subfigure}
    \caption{The dashed line represents the $\lambda$-stable perpendicular periodic orbit and the solid line the associated periodic orbit for $\Phi_\lambda$ with $\lambda=0.9$. }\label{fig:454590}
\end{figure}

Doing the same computations as in Proposition~\ref{prop-equilateral} and Proposition~\ref{306090}, one directly checks that these periodic cylinders are $\lambda$-stable (See Figure~\ref{fig:454590}).

\hfill$\Box$

%

\section*{Acknowledgements}
JPG was partially supported by Funda\c c\~ao para a Ci\^encia e a Tecnologia (FCT/MEC) through the grant SFRH/BPD/78230/2011 and
 partially supported by the Project CEMAPRE - UID/MULTI/00491/2013 financed by FCT/MEC through national funds.
 ST gratefully acknowledges the support of R\'egion PACA Project APEX ``Syst\`emes dynamiques: Probabilit\'es et Approximation Diophantienne PAD''

The authors wish to express their gratitude to Julien Cassaigne for useful discussions about Sturmian sequences.


\bibliographystyle{amsplain}
\bibliography{Bibfile}

\end{document}